\documentclass[12pt,leqno]{article}
\usepackage[shortlabels]{enumitem}
\usepackage[utf8]{inputenc} 
\usepackage[T1]{fontenc}
\usepackage{preamble}
\usepackage{tikz-cd}

\title{RC-positivity of complex surfaces 
}
\author{You-Cheng Chou and Kuang-Ru Wu}

\usepackage{fullpage}
\usepackage{multicol}

\begin{document}
\date{}

\parskip=6pt

\maketitle

\begin{abstract}
Given a compact complex surface with a nonflat Ricci-flat K\"ahler metric, we show that its holomorphic tangent bundle admits an RC-positive Hermitian metric. The proof relies on a characterization of RC-positivity through the anti-self-dual part of the Weyl operator, the Weitzenb\"ock formula on the four dimensional Einstein manifold, and a conformal perturbation on the  Ricci-flat metric. 

As a consequence, we prove that RC-positivity is not preserved after taking tensor, exterior, or symmetric power. Moreover, we show that RC-positivity is a strictly weaker notion than uniform RC-positivity, and that RC-positivity of holomorphic tangent bundle does not necessarily imply rational connectedness of the base manifold. 
\end{abstract}

\section{Introduction}

The notion of RC-positivity, introduced by Yang \cite{YangCamb}, has proven to be important and effective in the study of complex geometry. One prominent case is Yang's proof on a conjecture of Yau on projectivity and rational connectedness of a compact K\"ahler manifold with positive holomorphic sectional curvature (for developments and applications of RC-positivity, see \cite{YangCompos, YangForum, YangJussieu, Yangjdg2024, xiong2024rc, wu2025uniform, wu2026uniform, QSZ}).

We call a holomorphic vector bundle RC-positive if it admits a Hermitian metric that is RC-positive (the precise definition will be recalled in Section \ref{sec 2}). Compared to other positivity notions in geometry such as Griffiths positivity, Nakano positivity, and ampleness, RC-positivity is less developed. For example, it is not known if the tensor, exterior, or symmetric power of an RC-positive holomorphic vector bundle is still RC-positive. Somewhat surprisingly, the answer to this question is in the negative, and it is a consequence of our main result:

\begin{theorem}\label{main thm}
    If $X$ is a compact complex surface with a nonflat Ricci-flat K\"ahler metric, then the holomorphic tangent bundle $T^{1,0}X$ admits an RC-positive Hermitian metric.
\end{theorem}

Specifically, because K3 surfaces admit nonflat Ricci-flat K\"ahler metrics, we deduce the following corollary.

\begin{corollary}\label{cor}
    For a K3 surface $X$, the holomorphic tangent bundle $T^{1,0}X$ admits an RC-positive Hermitian metric. However, the exterior power $\bigwedge^2 T^{1,0}X$, the tensor power $\bigotimes^k T^{1,0}X$, and the symmetric power $S^kT^{1,0}X$, for $k\geq 2$, do not admit any RC-positive Hermitian metric. 
\end{corollary}
In Corollary \ref{cor}, the reason that $\bigwedge^2 T^{1,0}X=K^{-1}_X$ does not admit any RC-positive Hermitian metric is because $K^{-1}_X$ is the trivial line bundle (see Lemma \ref{line bundle}). That $S^kT^{1,0}X$ is not RC-positive for $k\geq 2$ is less direct, and will be proved in Section \ref{Section 3}. Since $S^k T^{1,0}X$ is the quotient of $\otimes^kT^{1,0}X$ and taking quotient preserves RC-positivity by \cite[Theorem 3.5]{YangCamb}, $\otimes^kT^{1,0}X$ is not RC-positive for $k\geq 2$.

As a consequence, the tensor, exterior, or symmetric power of an RC-positive bundle is not necessarily RC-positive. Corollary \ref{cor} also answers negatively a question of Yang \cite[Problem 4.15]{YangForum}, which asks whether the following are equivalent on a projective manifold $X$.
\begin{enumerate}
\renewcommand{\labelenumi}{(\arabic{enumi})}
\renewcommand{\theenumi}{\arabic{enumi}}

\item\label{cond:uniform-rc} $T^{1,0}X$ is uniformly RC-positive.
    
\item\label{cond:rc} $T^{1,0}X$ is RC-positive.

\item\label{cond:rc-connected} $X$ is rationally connected.
\end{enumerate}

According to the work of Yang \cite[Theorem 1.3]{YangForum}, 
(\ref{cond:uniform-rc}) implies both (\ref{cond:rc}) and 
(\ref{cond:rc-connected}). However, in view of Corollary \ref{cor}, we have 
\begin{corollary}\label{cor 2}
    (\ref{cond:rc}) does not imply (\ref{cond:uniform-rc}) or (\ref{cond:rc-connected})
\end{corollary}
Indeed, the holomorphic tangent bundle of a  projective K3 surface is RC-positive by Corollary \ref{cor}, but a K3 surface satisfies \(H^0(X,\Omega^2_X)\simeq \mathbb{C}\), so it cannot be rationally connected. Hence (\ref{cond:rc}) does not imply (\ref{cond:uniform-rc}) or (\ref{cond:rc-connected}). Especially, 
RC-positivity is strictly weaker than uniform RC-positivity.

Regarding the remaining directions, it is known that, for a compact K\"ahler $X$, it is projective and rationally connected if and only if  $T^{1,0}X$ admits a Hermitian metric $h$ whose mean curvature $\sqrt{-1}\Lambda_\omega R^h$ is positive with respect to some Hermitian metric $\omega$ on $X$ by \cite[Theorem 1.7]{LiZhangZhang}. Such a Hermitian metric is RC-positive by \cite[Theorem 3.6]{YangCamb}, so (\ref{cond:rc-connected}) implies (\ref{cond:rc}).

Around the time we finished our manuscript, we learned about a recent preprint \cite{QSZ} where Qing, Sun, and Zhou prove that (\ref{cond:rc-connected}) implies (\ref{cond:uniform-rc}). They also show that the induced Fubini--Study metric on the Fermat quartic surface, which is a K3 surface, is RC-positive by a direct computation, and so (\ref{cond:rc}) does not imply (\ref{cond:uniform-rc}) or (\ref{cond:rc-connected}) as stated in our Corollary \ref{cor 2}. However, our results appear to demonstrate more examples than the Fermat quartic surface, and our approach is different from theirs as we explain below.

The proof of Theorem \ref{main thm} can be separated as two parts. First, given a Ricci-flat K\"ahler metric $h$ on a compact complex surface, we characterize RC-positivity of $h$ in terms of invertibility of the anti-self-dual part of the Weyl operator (Lemma \ref{RC criterion 2}), and then  we use  the Weitzenb\"ock formula on the four dimensional Einstein manifold due to Derdzi\'nski   (\cite{Derdzinski,besse1987einstein,Weitzenbockformula}) to show the existence of a nonempty open set in $X$ where the metric $h$ is RC-positive. Secondly, we conformally perturb $h$ to get a new metric that is RC-positive everywhere in $X$. While the second part of the proof is independent of the dimension of $X$, the first part relies heavily on the four dimension geometry. 

Let us come back to Yang's question \cite[Problem 4.15]{YangForum}. There is a more general version concerning an arbitrary holomorphic vector bundle $E$ rather than the holomorphic tangent bundle $T^{1,0}X$. To that end, we define  another positivity notion, which we mentioned earlier already, lying between uniform RC-positivity and RC-positivity. Let $(E,h)$ be a Hermitian holomorphic vector bundle and denote the Chern curvature by $R^h.$ Let $\omega$ be a Hermitian metric on $X$, and take the trace of the Chern curvature $R^h$ with respect to $\omega$ to get $\sqrt{-1}\Lambda_\omega R^{h}\in\End E$. Following \cite[Section 3.1]{reprintDiffofcomplexbundles}, we call $\sqrt{-1}\Lambda_\omega R^{h}$ the mean curvature of $(E,h)$ with respect to $\omega$. If we choose $E=T^{1,0}X$ and $h=\omega$, then $\sqrt{-1}\Lambda_\omega R^{h}$ is the second Chern--Ricci curvature (see \cite[Page 188]{YangCamb}).

A holomorphic vector  bundle $E$ is called mean curvature positive if it admits a Hermitian metric $h$ whose mean curvature $\sqrt{-1}\Lambda_\omega R^{h}$ is positive with respect to some Hermitian metric $\omega$ on $X$. There is a relation:
\begin{equation*}
    E \text{ is uniformly RC-positive} \Rightarrow E \text{ is mean curvature positive} \Rightarrow E \text{ is RC-positive.}
\end{equation*}
The first implication is in \cite[Proposition 3.6]{LiZhangZhang}, and the second is in \cite[Theorem 3.6]{YangCamb}. By Corollary \ref{cor}, RC-positivity does not imply mean curvature positivity; this is because mean curvature positivity implies RC-positivity for all the exterior powers by \cite[Theorem 3.6]{YangCamb}. The remaining question is whether the converse of the first direction is true, but we do not have an answer now. Nevertheless, as we have mentioned already, when $E=T^{1,0}X$ and $X$ is K\"ahler, mean curvature positivity of $T^{1,0}X$ implies that $X$ is projective and rationally connected by \cite[Theorem 1.7]{LiZhangZhang}, which implies $T^{1.0}X$ is uniformly RC-positive by the recent work in \cite{QSZ}.

We also note that when $\rank E=1$, RC-positivity implies uniform RC-positivity by definition, so the three positivity notions are equivalent.

Let us mention one more positivity notion related to RC-positivity.
Denote by $O_{P(E^*)}(1)$ the hyperplane section bundle over the projectivized bundle $P(E^*)$, where $E^*$ is the dual bundle of $E$. A bundle $E$ is called weakly RC-positive if there is a metric $h$ on $O_{P(E^*)}(1)$ whose curvature $R^h$ is positive on the fiber $P(E^*_x)$ for any $x\in X$ and $R^h$ has at least $r$ positive eigenvalues at every point in $P(E^*)$.

In \cite[Question 7.11]{YangCamb}, Yang asked whether weak RC-positivity of $E$ implies RC-positivity of $E$ (the converse is true, \cite[Proposition 4.1]{YangCamb}). In light of this question's similarity with a conjecture of Griffiths and Berndtsson's results in \cite{Berndtsson09}, it is natural to ask whether the following are true: 
\begin{enumerate}
\renewcommand{\labelenumi}{(\arabic{enumi})}
\renewcommand{\theenumi}{\arabic{enumi}}
    \item\label{det} Weak RC-positivity of $E$ implies RC-positivity of $S^kE\otimes \det E$ for any $k\geq 0$.
    
    \item\label{no det} Weak RC-positivity of $E$ implies RC-positivity of $S^kE$ for large $k$.
\end{enumerate}

\begin{corollary}
    (\ref{det}) and (\ref{no det}) are both false.
\end{corollary}

Indeed, for (\ref{det}),   taking $k=0$, $ E=T^{1,0}X$, and $\dim X=2$, we are basically asking whether weak RC-positivity of $T^{1,0}X$ implies RC-positivity of $ \det T^{1,0}X=\bigwedge^2 T^{1,0}X$, and the answer is no by Corollary \ref{cor}. For (\ref{no det}), we use again Corollary \ref{cor} to conclude it is false.

The second named author would like to thank Po-Wei Lee for discussions and encouragements. Part of the research was done while the second named author was visiting KIAS, and he thanks the institute for the excellent working environment and hospitality. The authors benefit from conversations with ChatGPT 5.6 Sol, especially the crucial connection between four dimension geometry and RC-positivity. The research is partially supported by National Science and Technology Council, grant number 115-2115-M-002-004-MY3.

\section{Preliminaries}\label{sec 2}

\subsection{RC-positivity of holomorphic vector bundles}

Let $X$ be a complex manifold of dimension $n$ and $E\to X$ a holomorphic vector bundle of rank $r$ equipped with a  Hermitian metric $h$.  
The Chern connection $\nabla$ of $(E,h)$ is the unique connection that is compatible with $h$ and the complex structure on $E$. The Chern curvature $R^{h}:=\nabla^2$ is a $(1,1)$-form with values in $\End E$,
$$
R^{h}\in C^{\infty}_{1,1} (X,\End E).
$$
Let $\{z_i\}_{i=1}^n$ be local coordinates in $X$ and $\{e_\alpha\}_{\alpha=1}^r$ be a local holomorphic frame of $E$. Define $h_{\alpha\bar\beta}=\langle e_\alpha,e_\beta\rangle_h$. In terms of local coordinates, the Chern curvature satisfies 
$$R^{h}=\sum R^\gamma_{i\bar{j}\alpha}dz^i\wedge d\bar{z}^j\otimes e^\alpha\otimes e_\gamma,$$
where $\{e^\alpha\}$ is the dual frame on $E^*$ and we use the isomorphism $\Hom(E,E)\simeq E^*\otimes E$; moreover, $R^\gamma_{i\bar{j}\alpha}=h^{\gamma\bar{\beta}}R_{i\bar{j}\alpha\bar{\beta}}$
with $$R_{i\bar{j}\alpha\bar{\beta}}:=-\frac{\partial h_{\alpha\bar{\beta}}}{\partial z_i \partial \bar{z}_j}+h^{\gamma \bar{\delta}}\frac{\partial h_{\alpha \bar{\delta}} }{\partial z_i}\frac{\partial h_{\gamma\bar{\beta}}}{\partial \bar{z}_j}.$$
Let $T^{1,0}X$ be the holomorphic tangent bundle of $X$. For $u\in T^{1,0}_xX$ and $v\in E_x$, we consider the endomorphism part of $R^h$ acting on $v$ and then taking inner product with $v$ to get the $(1,1)$-form $\langle R^hv,v\rangle_h$. We also  consider the $(1,1)$-form part of $R^h$ acting on $u$ to get the endomorphism  $R^h(u,\bar{u})$. Of course, $\langle R^hv,v\rangle_h(u,\bar u)=\langle R^h(u,\bar{u})v,v\rangle_h$. Using local coordinates,     
$$
\langle R^hv,v\rangle_h(u,\bar u)=\langle R^h(u,\bar{u})v,v\rangle_h=\sum R_{i\bar{j}\alpha\bar{\beta}}u_i\bar{u}_jv_\alpha \bar{v}_\beta, 
$$
where $u=\sum u_i\partial/\partial z_i, v=\sum v_\alpha e_\alpha$.
\begin{definition}
\
\begin{enumerate}
\item A Hermitian metric $h$ is RC-positive at $x\in X$ if, for every nonzero $v\in E_x$, there is a nonzero $u\in T_x^{1,0}X$ such that $\langle R^hv,v\rangle_h(u,\bar u)>0$.

\item A Hermitian metric $h$ is uniformly RC-positive at $x\in X$ if there is a nonzero
$u\in T_x^{1,0}X$ such that $\langle R^hv,v\rangle_h(u,\bar u)>0$ for every nonzero $v\in E_x$.
\end{enumerate}
\end{definition}

A Hermitian metric $h$ is called (uniformly) RC-positive if it is (uniformly) RC-positive
at every point of $X$. A holomorphic vector bundle $E$ is called (uniformly) RC-positive if it admits such a metric. It follows directly from the definitions that a uniformly RC-positive bundle is RC-positive. The converse need not hold, as Corollary \ref{cor 2} shows.

We begin with a lemma, which is a slight improvement of \cite[Proposition 4.2]{YangCamb}.
\begin{lemma}\label{line bundle}
    Let $E$ be a holomorphic vector bundle of rank $r$ over  a compact complex manifold $X$ of dimension $n$. If $E$ is weakly RC-positive, then $H^0(X, S^kE^*)=0$ for any $k\geq 1$. In particular, trivial bundles over a compact complex manifold cannot be RC-positive. 
\end{lemma}
\begin{proof}
    By \cite[Theorem 4.5]{YangCamb}, $H^0(X, S^lE^*)=0$ for large $l$. Now, fix $k \geq 1$ and let $s\in H^0(X, S^kE^*)$. Denote by $\mu_{l,k}$ the canonical map $S^l(S^kE^*)\to S^{lk}E^*$. For $f_1,\ldots,f_l\in S^kE^*$,  $\mu_{l,k}(f_1\odot \cdots \odot f_l)=f_1\cdots f_l$ where $\odot$ is the symmetric product in $S^l(S^kE^*)$ and $f_1\cdots f_l$ is the ordinary polynomial multiplication. We have $\mu_{l,k}(s^{\odot l})=s^l\in H^0(X, S^{lk}E^*)$. For large $l$, $H^0(X, S^{lk}E^*)=0$ by \cite[Theorem 4.5]{YangCamb}, so $s^l=0$. For $x\in X$, $s^l(x)=s(x)^l$ and $s(x)$ is a homogeneous polynomial of degree $k$ on $E_x$, so $s^l(x)=0$ implies $s(x)=0$. As a result, any $s\in H^0(X, S^kE^*)$ is zero, and hence   $H^0(X, S^kE^*)=0$ for any $k\geq 1$.
\end{proof}



We make a simple observation below.
\begin{lemma}\label{RC criterion'}
     Let $(E,h)$ be a Hermitian holomorphic vector bundle over a complex manifold $X$ of dimension $n$ and let $\omega$ be a Hermitian metric on $X$. Assume the mean curvature $\sqrt{-1}\Lambda_\omega R^h$ is nonnegative. The metric $h$ is RC-positive at $x\in X$ if and only if the $(1,1)$-form $\langle R^h v,v\rangle_h$ is nonzero for any nonzero $v\in  E_x$.
\end{lemma}
\begin{proof}
    The statement is pointwise, so $X$ does not have to be compact. The only if part is direct by the definition of RC-positivity. So, let us assume the $(1,1)$-form $\langle R^h v,v\rangle_h$ is nonzero for any nonzero $v\in  E_x$.
    Fix a nonzero $v\in  E_x$. If all the eigenvalues of the $(1,1)$-form $\langle R^h v,v\rangle_h$ are nonpositive, then all the eigenvalues have to be zero because    $\sqrt{-1}\Lambda_\omega\langle R^h v,v\rangle_h   =\langle  \sqrt{-1}\Lambda_\omega R^h v,v\rangle_h \geq 0$ by the assumption in the lemma. But this implies the $(1,1)$-form $\langle R^h v,v\rangle_h=0$, a contradiction. Therefore, the $(1,1)$-form $\langle R^h v,v\rangle_h$ has at least one positive eigenvalue, and hence $h$ is RC-positive at $x$. 
\end{proof}

The lemma below shows that a Hermitian metric, which is RC-positive on some open set $U\subset X$, can be made RC-positive on the entire $X$ after a conformal perturbation.
\begin{lemma}\label{lem conformal}
   Let $(E,h)$ be a Hermitian holomorphic vector bundle over a compact complex manifold $X$ of dimension $n$ and let $\omega$ be a Hermitian metric on $X$. Assume the mean curvature $\sqrt{-1}\Lambda_\omega R^h$ is nonnegative. If  $h$ is RC-positive on a nonempty open set $U\subset X$, then there exists $f\in C^{\infty}(X,\mathbb{R})$  such that the  conformal change $he^{-f}$ is RC-positive on $X$  
\end{lemma}

\begin{proof}
Let $\hat{\omega}$ be the Gauduchon metric conformal to $\omega$. So, $\hat{\omega}=e^u\omega$ for some $u\in C^\infty(X,\mathbb{R})$, and $\sqrt{-1}\Lambda_{\hat{\omega}} R^h=e^{-u}\sqrt{-1}\Lambda_\omega R^h$ is still nonnegative. Choose an open set $B$ whose  closure is contained in $U$. For $v\in E$, we have the $(1,1)$-form $\langle  R^h v, v\rangle_h$, and we  denote by $$\lambda_{\max, \hat{\omega}}(\langle  R^h v, v\rangle_h)$$ the largest eigenvalue of $\langle  R^h v, v\rangle_h$ with respect to $\hat{\omega}$. If $v \in E|_U$, then $\lambda_{\max, \hat{\omega}}(\langle  R^h v, v\rangle_h)>0$ because $h$ is RC-positive on $U$.

 Consider the continuous map $ \{v\in E|_{\overline{B}}:|v|_h=1 \}\ni v\mapsto \lambda_{\max, \hat{\omega}}(\langle  R^h v, v\rangle_h)$. As $\{v\in E|_{\overline{B}}:|v|_h=1 \}$ is compact, $\lambda_{\max, \hat{\omega}}(\langle  R^h v, v\rangle_h)$ achieves its minimum somewhere. The minimum, say $\delta$, is positive because $h$ is RC-positive on $\overline{B}\subset U$.

Let $\psi$ be a real-valued smooth function on $X$ supported in $B$, $\psi\geq 0$, and $\int_X \psi dV_{\hat{\omega}} >0$. Define
\begin{equation}
    \rho := 1- \frac{ \int_X dV_{\hat{\omega}} }{ \int_X \psi dV_{\hat{\omega}} } \psi.
\end{equation}
Then, $\rho = 1$ on $X \setminus B$. Moreover, $\int_X \rho dV_{\hat{\omega}}=0$, so there is  $f\in C^{\infty} (X,\mathbb{R})$ such that $\Delta_{\hat{\omega}} f=\Lambda_{\hat{\omega}}\sqrt{-1}\partial\bar{\partial}f = \rho$ (see \cite{Gauduchon} and \cite[Theorem 2.2]{CTW}).

Define $h_\varepsilon=e^{-\varepsilon f}h$ for $\varepsilon>0$. Because $\sqrt{-1}R^{h_\varepsilon}=\sqrt{-1}R^h+\varepsilon\sqrt{-1}\partial\bar{\partial}f\otimes \Id_E$, we have, for $v\in E$,
\begin{equation}\label{11form}
    \langle \sqrt{-1}R^{h_\varepsilon} v, v\rangle_{h_\varepsilon}=e^{-\varepsilon f}\big( \langle \sqrt{-1}R^{h} v, v\rangle_{h}+|v|^2_h \varepsilon\sqrt{-1}\partial\bar{\partial}f \big).
\end{equation}
We discuss the case $X\setminus B$ first. Taking $\Lambda_{\hat{\omega}}$ on (\ref{11form}) yields
\begin{equation}
    \langle \sqrt{-1}\Lambda_{\hat{\omega}} R^{h_\varepsilon} v, v\rangle_{h_\varepsilon}=e^{-\varepsilon f}\big( \langle \sqrt{-1}\Lambda_{\hat{\omega}} R^{h} v, v\rangle_{h}+|v|^2_h \varepsilon\sqrt{-1}\Lambda_{\hat{\omega}}\partial\bar{\partial}f \big)\geq e^{-\varepsilon f} |v|^2_h \varepsilon \rho=e^{-\varepsilon f} |v|^2_h \varepsilon, 
\end{equation}
where the inequality is due to the fact $\sqrt{-1}\Lambda_{\hat{\omega}} R^h \geq 0$ and $\Delta_{\hat{\omega}} f=\rho$, and the last equality is due to $\rho=1$ on $X\setminus B$. Therefore, for any $v\neq 0$, $\langle \sqrt{-1}\Lambda_{\hat{\omega}} R^{h_\varepsilon} v, v\rangle_{h_\varepsilon}>0$, so $h_\varepsilon$ is RC-positive over $X\setminus B$ for any $\varepsilon >0$.

It remains to prove RC-positivity on $\overline{B}$. Let $C=\sup_X|\sqrt{-1}\partial\bar{\partial}f|_{\hat{\omega}}$. For $v\in E|_{\overline{B}}$ and $|v|_h=1$, we obtain from (\ref{11form}) the following estimate
\begin{equation}
    \lambda_{\max,{\hat{\omega}}}(\langle  R^{h_\varepsilon} v, v\rangle_{h_\varepsilon})\geq e^{-\varepsilon f}\big[  \lambda_{\max,{\hat{\omega}}}(\langle  R^{h} v, v\rangle_{h})-\varepsilon C    \big]\geq e^{-\varepsilon f }(\delta-\varepsilon C). 
\end{equation}
Because $\delta$ and $C$ are independent of $\varepsilon$, we can make $\varepsilon$ small so that $(\delta-\varepsilon C)$ becomes positive. Therefore, for any nonzero $v\in E|_{\overline{B}}$,  the $(1,1)$-form $\langle  R^{h_\varepsilon} v, v\rangle_{h_\varepsilon}$ has at least one positive eigenvalue, so $h_\varepsilon$ is RC-positive on $\overline{B}$. Together with the previous case, we have proved that $h_\varepsilon$ is RC-positive on $X$.
\end{proof}

Next, we consider the special case where $E$ is the holomorphic tangent bundle $T^{1,0}X$ of $X$.  Denote the real tangent bundle by $TX$ and the complex structure by $J$. We have the eigenspace decomposition $TX\otimes \mathbb{C}=T^{1,0}X\oplus T^{0,1}X$ and the identification $\Phi: TX \to  T^{1,0}X$ given by
\begin{equation}
 U\mapsto \Phi(U):=\frac{1}{2}(U-\sqrt{-1}JU).   
\end{equation}

Let $h$ be a Hermitian metric on $T^{1,0}X$. Then $\Phi^*h$ is a Hermitian metric on $(TX,J)$ and induces a $J$-invariant  Riemannian metric $g$ on $TX$ by setting $\langle U,V\rangle_g:=2 \Re{\langle U,V\rangle_{\Phi^* h}}$  for $U,V\in TX$.

If we furthermore assume $h$ is K\"ahler, then the Levi-Civita connection of  $g$ coincides with the Chern connection of $\Phi^*h$ (for example, \cite[Proposition 11.8]{Lecturesonkahler}). We denote the Riemannian curvature operator of $g$ by 
$\mathcal{R}:\wedge^2TX\longrightarrow\wedge^2TX$, so 
\begin{equation}
    \langle R_g(S,T)U,V\rangle_g=\langle  \mathcal{R}(S\wedge T), U\wedge V\rangle_g,
\end{equation}
where $S,T,U,V\in TX$ and  $R_g$ is the Riemannian curvature of $g$. Let $\flat: \wedge ^2TX\to \wedge ^2T^*X$ be the canonical isomorphism induced by $g$. 
\begin{lemma}\label{relation}
\begin{equation}
\langle R^hv,v \rangle_h= \frac{\sqrt{-1}}{2}(\mathcal{R}(V\wedge JV))^\flat, \text{ with } v=\Phi(V) \text{ and } V\in TX.    
\end{equation}    
\end{lemma}
\begin{proof}
Because $g= 2\Re \Phi^* h$, we have $\langle S,T\rangle_{\Phi^*h}=  [\langle S,T\rangle_g-\sqrt{-1}\langle JS,T\rangle_g]/2$ for $S,T\in TX$. Therefore, for $S$ and $T\in TX$, 
\begin{equation}\label{1}
\begin{aligned}    
 &\langle R^hv,v \rangle_h(S,T)=\langle R^{\Phi^*h} V,V \rangle_{\Phi^*h} (S,T)=\langle R_{g} V,V \rangle_{\Phi^*h} (S,T)=\langle R_{g}(S,T) V,V \rangle_{\Phi^*h} \\
 =&\frac12 [\langle R_{g}(S,T) V,V \rangle_{g} -\sqrt{-1}\langle R_{g}(S,T) JV,V\rangle_g]
 =-\frac12\sqrt{-1}\langle R_{g}(S,T) JV,V\rangle_g ,
\end{aligned}
\end{equation}
where in the fourth equality we use $J(R_g(S,T)V)=R_g(S,T)JV$ by \cite[(2.42)]{besse1987einstein} and in the last equality we use the fact $\langle R_{g}(S,T) V,V \rangle_{g}$ is skew-symmetric in the last two components.

On the other hand, 
\begin{equation}\label{2}
    (\mathcal{R}(V\wedge JV))^\flat(S,T)=\langle S\wedge T, \mathcal{R}(V\wedge JV) \rangle_g=\langle R_g(S,T)V,JV\rangle_g.
\end{equation}
Comparing (\ref{1}) and (\ref{2}), we obtain     $\langle R^hv,v \rangle_h=\sqrt{-1}(\mathcal{R}(V\wedge JV))^\flat/2$.
\end{proof}

\begin{lemma}\label{RC criterion}
    If $X$ is a complex manifold with a K\"ahler metric $h$ whose Ricci curvature is nonnegative, then $h$ is RC-positive at $x$ if and only if $\mathcal{R}(V\wedge JV)$ is nonzero  for any nonzero $V\in T_xX$.
\end{lemma}
\begin{proof}
This lemma is a special case of Lemma \ref{RC criterion'}. We have the holomorphic tangent bundle $T^{1,0}X$ with the K\"ahler metric $h$ and we denote by $\omega$ the fundamental form associated with $h$. Since $h$ is K\"ahler, the mean curvature or the second Chern--Ricci curvature $\sqrt{-1}\Lambda_\omega R^h$ is the Ricci curvature. So, by Lemma \ref{RC criterion'},  $h$ is RC-positive at $x$ if and only if $\langle R^h v,v\rangle_h$ is nonzero for any nonzero $v\in T^{1,0}_xX$. Using Lemma \ref{relation}, the last statement is equivalent to that $\mathcal{R}(V\wedge JV)$ is nonzero  for any nonzero $V\in T_xX$.
\end{proof}

We will use Lemma  \ref{RC criterion} for complex surfaces. But let us first recall some basics in four dimension geometry.

\subsection{Four dimension geometry and the Weyl operator}
Let $(M^4,g)$ be an oriented Riemannian four-manifold. Let $*$ be the Hodge star operator. On two-forms in the four-manifold, one has $*:\wedge ^2T^*M\longrightarrow\wedge^2T^*M$ and  $*^2=\Id$. Thus $\wedge^2 T^*M$ splits into the eigenspaces of $*$:
\begin{equation}\label{eq:Hodge-splitting}
\wedge ^2 T^*M=\wedge^+ T^*M\oplus\wedge ^- T^*M, \text{ where }
\wedge^\pm T^*M=\{\eta\in \wedge^2 T^*M:* \eta=\pm\eta\}.
\end{equation}
Each of $\wedge^+ T^*M$ and $\wedge^- T^*M$ has real rank three. Their sections are
called self-dual and anti-self-dual two-forms, respectively.

By \cite[(1.116)]{besse1987einstein}, the Riemannian curvature tensor $R_g$ of $g$, as an element in $S^2(\wedge^2T^*M)$, has the decomposition  
\begin{equation}\label{curvature decomp}
    R_g=\frac{s}{24}g\owedge g+ \frac{1}{2}(r-\frac{s}{4}g)\owedge g +W,
\end{equation}
where $s$ is the scalar curvature, $r$ is the Ricci tensor, and $W$ is the Weyl tensor. The notation $\owedge$ is the Kulkarni--Nomizu product on symmetric 2-tensors (for the precise definition, see \cite[Definition 1.110]{besse1987einstein}).

The Weyl tensor $W\in S^2(\wedge^2T^*M)$
can be viewed  as a symmetric operator from $\wedge^2 T^*M$ to $\wedge^2 T^*M$, and with respect to the eigenspace decomposition (\ref{eq:Hodge-splitting}), $W$ is block diagonal with $W^+$ and $W^-$ on the diagonal,  where $W^+$ is a traceless symmetric linear map from $\wedge^+ T^*M $ to $ \wedge^+ T^*M$ and $W^-$ a traceless symmetric linear map from $\wedge^- T^*M $ to $ \wedge^- T^*M$. $W^+$ and $W^-$ are called the self-dual and anti-self-dual parts of $W$, respectively (\cite[(1.127) and (1.128)]{besse1987einstein}). We will write $W=W^+\oplus W^-$ for brevity. We make three observations:
\begin{enumerate}
    \item If the Riemannian metric $g$ is Ricci-flat, then the decomposition (\ref{curvature decomp}) is reduced to $R_g=W=W^+\oplus W^-$. 
    
    \item If $(M,g)$ is a K\"ahler manifold of real dimension four, oriented in the natural way, then according to \cite[Proposition 16.62]{besse1987einstein}, $W^+=0$ provided that the scalar curvature $s=0$.

    \item Combing the above two remarks, if $M$ is a complex surface with a Ricci-flat K\"ahler metric $g$ and with the natural orientation, then $R_g=0\oplus W^-$.
\end{enumerate}

The following lemma is perhaps known to the experts, but we still write out the detials for completeness.
\begin{lemma}\label{surjection}
   Let $M$ be a complex surface with a Hermitian metric $\omega$. For $x\in M$, let  $(\wedge^{1,1}_0)_x$ be the space of the real $(1,1)$-forms at $x$ that are orthogonal to $\omega$.  The map $\mathcal{H}:T_x^*M\to (\wedge^{1,1}_0)_x$ defined by
   \begin{equation}\label{the map}
       \xi \mapsto (\xi\wedge J\xi-\frac{|\xi|^2}{2}\omega  )
\end{equation} is surjective. Moreover, if $\xi$ is nonzero, then $\xi\wedge J\xi -|\xi|^2\omega/2$ is nonzero.
\end{lemma}
\begin{proof}
First, the real 2-form $\xi\wedge J\xi$ is $J$-invariant, so it is a real $(1,1)$-form. Thus, $(\xi\wedge J\xi-|\xi|^2\omega/2  )$ is a real $(1,1)$-form. To see that $(\xi\wedge J\xi-|\xi|^2\omega/2  )$ is orthogonal to $\omega$, one notice that $\langle\omega,\omega\rangle=n=2$ and $\langle\xi\wedge J\xi, \omega\rangle=\omega(\xi^\#,J\xi^\#)=g(J\xi^\#, J\xi^\#)=|\xi^\#|^2=|\xi|^2 $, where $\#:T^*M\to TM$ is the canonical map induced by the metric $g$. Hence, $(\xi\wedge J\xi-|\xi|^2\omega/2  )$ is indeed in $(\wedge^{1,1}_0)_x$.

One way to prove that $\mathcal{H}$ is surjective is through local coordinates, and we will see that the map $\mathcal{H}$ is basically the Hopf map $S^3\to S^2$. 

Let the local coordinates around $x$ be $\{z_1=x_1+\sqrt{-1}y_1, z_2=x_2+\sqrt{-1}y_2\}$ such that $Jdx_j=dy_j$, $Jdy_j=-dx_j$. At $x$, we assume $\omega=2(dx_1\wedge dy_1+dx_2\wedge dy_2)$, so $\{dx_1,dy_1,dx_2,dy_2\}$ is an orthogonal basis of $T_x^*M$ and each vector has length $1/\sqrt{2}$. Assume $\xi=adx_1+bdy_1+cdx_2+ddy_2$ with $a,b,c,d\in \mathbb{R}$. A direct computation gives 
\begin{equation}\label{local}
\begin{aligned}
    \xi\wedge J\xi-\frac{|\xi|^2}{2}\omega=&\frac{a^2+b^2-c^2-d^2}{2}(dx_1\wedge dy_1-dx_2\wedge dy_2)\\+&(ac+bd)(dx_1\wedge dy_2-dy_1\wedge dx_2)\\+&(bc-ad)(dx_1\wedge dx_2+dy_1\wedge dy_2).
    \end{aligned}
\end{equation}
Meanwhile, by \cite[(2.3)]{Itoh}, 
$$
\begin{cases}
    e_1=:(dx_1\wedge dy_1-dx_2\wedge dy_2),\\ e_2=:(dx_1\wedge dy_2-dy_1\wedge dx_2),\\
    e_3=:(dx_1\wedge dx_2+dy_1\wedge dy_2)
\end{cases}
$$
is an orthogonal basis for the space $(\wedge^{1,1}_0)_x$ and each vector has length $1/\sqrt{2}$. If we write $\zeta_1=a+\sqrt{-1}b$ and $\zeta_2=c+\sqrt{-1}d$, then (\ref{local}) becomes 
\begin{equation}
    \xi\wedge J\xi-\frac{|\xi|^2}{2}\omega=\frac{(|\zeta_1|^2-|\zeta_2|^2)}{2}e_1+\Re{(\zeta_1\bar{\zeta}_2)}e_2+\Im{(\zeta_1\bar{\zeta}_2)}e_3.
\end{equation}
Therefore, if we  identify $T^*_xM$ with $\mathbb{R}^4$ using the basis $\{dx_1,dy_1,dx_2,dy_2\}$ and $(\wedge^{1,1}_0)_x$ with $\mathbb{R}^3$ using the basis $\{e_1/2, e_2/2, e_3/2\}$, then the map $\mathcal{H}:\mathbb{R}^4\to \mathbb{R}^3$
is \begin{equation}
  (a,b,c,d)\mapsto  \big(|\zeta_1|^2-|\zeta_2|^2, 2\Re{(\zeta_1\bar{\zeta}_2)}, 2\Im{(\zeta_1\bar{\zeta}_2)} \big),
\end{equation}
where $\zeta_1=a+\sqrt{-1}b$ and $\zeta_2=c+\sqrt{-1}d$. When $\mathcal{H}$ is restricted to $S^3\subset \mathbb{R}^4$, we obtain the Hopf map from $S^3$ to $S^2$. In particular, $\mathcal{H}$ maps the sphere of radius $1/\sqrt{2}$ in $T^*_x M$ onto the sphere of radius $1/\sqrt{8}$ in $(\wedge^{1,1}_0)_x$. Because $\mathcal{H}$ satisfies $\mathcal{H}(\lambda\xi)=\lambda^2\mathcal{H}(\xi)$ for any $\lambda>0$ and $\xi \in T^*_x M$, we get that $\mathcal{H}:T^*_x M\to (\wedge^{1,1}_0)_x$ is surjective.

Finally, for $\xi\neq 0$, $\mathcal{H}(\frac{\xi}{\sqrt{2}|\xi|})$ is on the sphere of radius $1/\sqrt{8}$, so it is nonzero. But $$\mathcal{H}(\frac{\xi}{\sqrt{2}|\xi|})=\frac{\mathcal{H}(\xi)}{2|\xi|^2},$$ so $\mathcal{H}(\xi)=\xi\wedge J\xi -|\xi|^2\omega/2$ is nonzero.
\end{proof}

Note that $\dim_\mathbb{C} M=2$ is essential in Lemma \ref{surjection}. In fact, for a complex manifold $M$ of dimension $n$, we can similarly define the map $\mathcal{H}$ 
\begin{equation}
    T_x^*M   \ni \xi \mapsto (\xi\wedge J\xi-\frac{|\xi|^2}{n}\omega  )\in (\wedge^{1,1}_0)_x.
   \end{equation}
But this map cannot be surjective for $n\geq 3$. This is because its differential $d_\xi \mathcal{H}$ has rank at most $\dim_\mathbb{R}T_x^*M=2n$, and $2n<n^2-1=\dim_\mathbb{R}(\wedge^{1,1}_0)_x$ for $n\geq 3$, hence the entire $T_x^*M$ is the critical set and $\mathcal{H}(T_x^*M)$ has measure zero in $(\wedge^{1,1}_0)_x$ by Sard's lemma.

We recall the following lemma from \cite[Lemma 2.1]{Itoh} and \cite[Page 429]{AHS}). 
\begin{lemma}\label{Itoh}
For a complex surface $M$ with a Hermitian metric $\omega$, the following hold for each $x\in M$
\begin{enumerate}
    \item A real 2-form $\alpha$ is self-dual if and only if $\alpha=\beta+\bar{\beta}+b\omega$ for a $(2,0)$-form $\beta$ and a real number $b$.
    \item A real 2-form $\alpha$ is anti-self-dual if and only if $\alpha$ is a real $(1,1)$-form orthogonal to $\omega$. That is, $(\wedge^- T^*M)_x=(\wedge^{1,1}_0)_x$ for any $x\in M$.
\end{enumerate}
\end{lemma}

Using Lemma \ref{RC criterion}, Lemma \ref{surjection}, and Lemma \ref{Itoh}, we deduce the following characterization of RC-positivity in terms of the invertibility of $W^-$.
\begin{lemma}\label{RC criterion 2}
    If $M$ is a complex surface with a Ricci-flat K\"ahler metric $g$, then   $g$ is RC-positive at $x$ if and only if the anti-self-dual part of the Weyl operator $$W^-:(\wedge^-T^*M)_x\longrightarrow (\wedge^-T^*M)_x$$ is invertible.
\end{lemma}
\begin{proof}
By Lemma \ref{RC criterion}, $g$ is RC-positive at $x$ if and only if  $\mathcal{R}(V\wedge JV)$ is nonzero  for any nonzero $V\in T_xM$. We use the canonical isomorphism between $T_xM$ and $T_x^*M$ induced by the metric $g$, so the statement is equivalent to that  $R_g(\xi\wedge J\xi)$ is nonzero  for any nonzero $\xi\in T_x^*M$

For $\xi \in T_x^*M$, the decomposition of the real two-form $\xi\wedge J\xi$ into its self-dual and anti-self-dual parts is 
\begin{equation}
    \xi\wedge J\xi=\frac{|\xi|^2}{2}\omega+(\xi\wedge J\xi -\frac{|\xi|^2}{2}\omega)\in \wedge^+ T^*M\oplus\wedge^- T^*M.
\end{equation}
This is because of Lemma \ref{Itoh}.

By the remarks before Lemma \ref{surjection}, we know $R_g=0 \oplus W^-$.   So $R_g(\xi\wedge J\xi )=W^-(\xi\wedge J\xi -|\xi|^2\omega/2)$. As a result,  $R_g(\xi\wedge J\xi)$ is nonzero  for any nonzero $\xi\in T_x^*M$ if and only if $W^-(\xi\wedge J\xi -|\xi|^2\omega/2)$ is nonzero for any nonzero $\xi\in T_x^*M$. This last statement is also equivalent to $W^-$ being invertible.  This is because $\xi\wedge J\xi -|\xi|^2\omega/2$ fills the entire $(\wedge^{1,1}_0)_x=(\wedge^- T^*M)_x$ by Lemma \ref{surjection}.
\end{proof}

We recall the following Weitzenb\"ock formula. 
This formula is originally due to Derdzi\'nski \cite{Derdzinski}. We use the formulation from \cite[Theorem 1.1]{Weitzenbockformula}. 

\begin{lemma}\label{lem Weitzenbock}
Let $(M,g)$ be an oriented Einstein four-manifold. The self-dual and anti-self-dual parts of the Weyl operator satisfy
\begin{equation}
\Delta |W^\pm|^2
=
2|\nabla W^\pm|^2
+
s|W^\pm|^2
-
36\det W^\pm.
\end{equation}
\end{lemma}

Using Lemma \ref{lem Weitzenbock}, we show that the set where $W^-$ is invertible is nonempty.
\begin{lemma}\label{proper set}
  If $M$ is a compact complex surface with a nonflat Ricci-flat K\"ahler metric $g$, then the set 
\begin{equation}\label{eq:Z}
  Z:=\{x\in M:\det W^-(x)=0\}
\end{equation}
is a proper subset of $M$.
\end{lemma}
\begin{proof}
Using the Weitzenb\"ock formula in Lemma \ref{lem Weitzenbock} for $W^-$ and the fact $s=0$, we obtain
\begin{equation}\label{eq:weyl-ricci-flat}
  \Delta|W^-|^2
  =2|\nabla W^-|^2-36\det W^-.
\end{equation}

If $Z=M$, then $\det W^-=0$ everywhere on $M$. So after integrating
\eqref{eq:weyl-ricci-flat} over the compact manifold $M$ and using Stokes' theorem, we get  
\begin{equation}\label{eq:parallel-wminus}
  0=\int_M|\nabla W^-|^2\,dV_g.
\end{equation}
Hence, $\nabla W^-=0$. Meanwhile, by the remark before Lemma \ref{surjection}, we know that $R_g=0\oplus W^-$, so $\nabla R_g=0$ and $g$ is locally
symmetric. However, a locally symmetric Ricci-flat metric has to be flat, which contradicts the assumption in Lemma \ref{proper set}. Therefore, $Z\ne M$.
\end{proof}

\section{Proofs of the main results}\label{Section 3}

\begin{proof}[Proof of Theorem \ref{main thm}]
    Let $X$ be a compact complex surface with a nonflat Ricci-flat K\"ahler metric $g$. We want to construct an RC-positive Hermitian metric on the holomorphic tangent bundle $T^{1,0}X$. 

Let $W^-$ be the anti-self-dual part of the Weyl operator associated with the metric $g$. From Lemma \ref{proper set}, we know $X \setminus Z$ is a nonempty open set, where  $Z:=\{x\in X:\det W^-(x)=0\}$. By Lemma \ref{RC criterion 2}, $g$ is RC-positive on $X \setminus Z$. By Lemma \ref{lem conformal}, some conformal change of $g$ is RC-positive on $X$, and hence the theorem follows.
\end{proof}

The only part in Corollary \ref{cor} that we have not yet proved is that $S^kT^{1,0}X$ with $k\geq 2$ is not RC-positive where $X$ is a K3 surface. We need a lemma first.
\begin{lemma}\label{lem long} \
\begin{enumerate}
\item    Let $V$ be a complex vector space of dimension $2$ and let $k\geq 2$. Any $u\in S^kV$ can be written as $v_1\odot\cdots \odot v_k$ for $v_1, \ldots, v_k\in V$, where $\odot$ denotes the symmetric product in $S^kV$. For $u\neq 0$, the factorization is unique up to permutation and rescaling.  For $\sigma\in \wedge^2 V^*$,  the map $\Delta^k_\sigma: S^kV \to \mathbb{C}$ given by 
    \begin{equation}\label{Delta}
     \Delta^k_\sigma(u)=\prod_{1\leq i<j \leq k}\sigma(v_i,v_j)^2   
    \end{equation}
  is well-defined. Moreover, $\Delta^k_\sigma$ is a homogeneous polynomial of degree $2(k-1)$ on $S^kV$; that is,  $\Delta^k_\sigma\in S^{2(k-1) }(S^kV)^*$.  Finally, if $\sigma\neq0$, then $\Delta_\sigma^k\neq 0$.

  \item Let $X$ be a complex surface and $\sigma$ be a local holomorphic section of the bundle  $\wedge^2 (T^{1,0}X)^*$. Let $k\geq 2$. The section defined by
    \begin{equation}\label{sigma section}
       x\mapsto  \Delta^k_{\sigma(x)}\in S^{2(k-1)}(S^kT^{1,0}_xX)^*   
    \end{equation}
is holomorphic.
 \end{enumerate} 
  \end{lemma}
\begin{proof}
First, we show that any $u\in S^kV$ can be written as $v_1\odot\cdots \odot v_k$ for $v_1, \ldots, v_k\in V$. Because $\dim_\mathbb{C} V=2$, we can write $u\in S^kV$ as a homogeneous polynomial of degree $k$ with two variables $x$ and $y$,  
\begin{equation}\label{homogeneous}
    u=\sum_{j=0}^k a_j x^{k-j} y^j 
\end{equation}
for complex numbers $\{a_j\}_{j=0}^k$. By setting $x=t$ and $y=1$ in (\ref{homogeneous}), we obtain the polynomial $\sum_{j=0}^k a_j t^{k-j} $, which is of degree $m\leq k$, and we can factorize
\begin{equation}
    \sum_{j=0}^k a_j t^{k-j}=c(t-r_1)\cdots (t-r_m), \text{ where } c, r_1, \ldots, r_m\in \mathbb{C}.
\end{equation}
Therefore, 
\begin{equation}
 u=y^k\sum_{j=0}^k a_j \big(\frac{x}{y}\big)^{k-j}=y^kc(\frac{x}{y}-r_1)\cdots (\frac{x}{y}-r_m)=  y^{k-m}c(x-r_1 y)\cdots (x-r_my),  
\end{equation}
which is a product of homogeneous polynomials of degree $1$. Therefore, $u$ can be written as $v_1\odot\cdots \odot v_k$ for $v_1, \ldots, v_k\in V$. 

Suppose $u\neq 0$ has two factorizations $v_1\odot\cdots \odot v_k$ and $w_1\odot\cdots \odot w_k$. Using the view of homogeneous polynomials, we can write $v_j=a_jx+b_jy$ and $w_j=c_jx+d_jy$ with $a_j,b_j,c_j, d_j\in \mathbb{C}$ for each $j$. Because $\mathbb{C}[x,y]$ is a UFD, there exist a permutation $\pi$ of $\{1,\ldots, k\}$ and $\lambda_j\in \mathbb{C}$ for $j=1\sim k$ with $\Pi\lambda_j=1$ such that $v_j=\lambda_j w_{\pi(j)}$ for each $j$. So, the factorization of $u$ is unique up to permutation and rescaling. Moreover, we observe that
\begin{equation}
\begin{aligned}
    &\prod_{1\leq i<j \leq k}\sigma(v_i,v_j)^2=
    \prod_{1\leq i<j \leq k}\sigma(\lambda_i w_{\pi(i)},\lambda_j w_{\pi(j)})^2=
    \prod_{1\leq i<j \leq k}\lambda_i^2\lambda_j^2\sigma( w_{\pi(i)}, w_{\pi(j)})^2\\
    = &(\prod_{j=1}^k\lambda_j)^{2(k-1)} \prod_{1\leq i<j \leq k}\sigma( w_{\pi(i)}, w_{\pi(j)})^2   =\prod_{1\leq i<j \leq k}\sigma(w_i,w_j)^2  
\end{aligned}
\end{equation}
showing that the map $\Delta^k_\sigma: S^kV \to \mathbb{C}$ in (\ref{Delta}) is well-defined.

Next, we are going to show that $\Delta^k_\sigma$ is a homogeneous polynomial of degree $2(k-1)$ on $S^kV$. Let $\{e_0,e_1\}$ be a basis for $V$ and denote $\sigma (e_0,e_1)=s$. For $u\in S^kV $, we write $u=v_1\odot \cdots \odot v_k$ for $v_1,\ldots, v_k\in V$. Let $u=\sum_{j=0}^k a_j e_0^{k-j}e_1^j$ with $a_0,\ldots, a_k\in \mathbb{C}$, and  $v_j=b_je_0+c_je_1$ with $b_j,c_j\in\mathbb{C}$ for $j=1\sim k$. Hence,
\begin{equation}
    \sum_{j=0}^k a_j e_0^{k-j}e_1^j=u=v_1\odot \cdots \odot v_k=\bigodot_{j=1}^k (b_je_0+c_je_1).
\end{equation}
Note that $\sigma(b_ie_0+c_ie_1, b_je_0+c_je_1)=s(b_ic_j-c_ib_j)$, so 
\begin{equation}\label{product}
 \Delta_\sigma^k(u)=\prod_{1\leq i<j\leq k}s^2(b_ic_j-c_ib_j)^2=(s^2)^{\frac{k(k-1)}{2}}\prod_{1\leq i<j\leq k}(b_ic_j-c_ib_j)^2.   
\end{equation}
On the other hand, if we use the view of homogeneous polynomials, then
 $u$ corresponds to 
\begin{equation}
 F(x,y)=\sum_{j=0}^k a_j y^{k-j}x^j=\prod_{j=1}^k (b_jy+c_jx).  
\end{equation} 
The discriminant of $F$ is
\begin{equation}\label{disc'}
 D(F(x,y))=(-1)^{\frac{k(k-1)}{2}}\prod_{1\leq i<j\leq k}(b_ic_j-c_ib_j)^2.   
\end{equation}
Combing (\ref{product}) and (\ref{disc'}) yields 
\begin{equation}\label{delta and D}
     \Delta_\sigma^k(u)=(-s^2)^{\frac{k(k-1)}{2}}D(F).
\end{equation}

Let us assume $c_j\neq 0$ for any $j$ for the moment. If we denote $b_j/c_j$ by $r_j$ and 
denote $F(x,1)$ by $f(x)$, then $f(x)=\prod(b_j+c_jx)=\prod (r_j+x)\prod c_j$ and the discriminant of $f$ is
\begin{equation}\label{disc}
\begin{aligned}
 D(f(x))=&(-1)^{\frac{k(k-1)}{2}}(\prod_{j=1}^k c_j)^{2(k-1)}\prod_{1\leq i<j\leq k}(r_i-r_j)^2\\
 =&(-1)^{\frac{k(k-1)}{2}}\prod_{1\leq i<j\leq k} (c_ic_j)^{2}\prod_{1\leq i<j\leq k}(r_i-r_j)^2\\
 =& (-1)^{\frac{k(k-1)}{2}}\prod_{1\leq i<j\leq k} (b_ic_j-c_ib_j)^2=D(F(x,y)).   
\end{aligned}
\end{equation}
By \cite[Page 405, (1.30)]{GKZ}, because $f(x)=\sum_{j=0}^k a_j x^j$, $D(f(x))$ is a homogeneous polynomial of degree $2(k-1)$ in $a_0,\ldots, a_k$. So, $D(F(x,y))$ is also a homogeneous polynomial of degree $2(k-1)$ in $a_0,\ldots, a_k$, and the statement holds without assuming $c_j\neq 0$ for any $j$ because both sides are polynomials in $b_i$ and $c_i$. Therefore, from (\ref{delta and D}), we obtain that $\Delta^k_\sigma$ is a homogeneous polynomial of degree $2(k-1)$ on $S^kV$. 

Assuming $\sigma \neq 0$, we see that $\sigma(e_0,e_1)=s$ must be nonzero. If we define $u=\bigodot_{j=1}^k (b_je_0+e_1)$ with distinct $b_j$, then $D(F)\neq 0$ by (\ref{disc'}), and so $\Delta_\sigma^k(u)\neq 0$ by (\ref{delta and D}). We obtain $\Delta_\sigma^k\neq 0$.

For the second part of the lemma, it suffices to show that $x\mapsto \Delta^k_{\sigma(x)}(u(x))$ is holomorphic for any local holomorphic section $u$ of $S^kT^{1,0}X$. Consider a local holomorphic frame $\{e_0,e_1\}$ of $T^{1,0}X$. Because $\sigma$ is holomorphic, $\sigma(e_0,e_1):=s$ is holomorphic. For a local holomorphic section $u$ of $S^kT^{1,0}X$, if we write $u=\sum_{j=0}^k a_j e_0^{k-j}e_1^j$, then all the $a_j$ are holomorphic.  From (\ref{delta and D}), we see $$x\mapsto \Delta^k_{\sigma(x)}(u(x))=(-s^2)^{\frac{k(k-1)}{2}}D(F)$$ is holomorphic because $D(F)$ is a homogeneous polynomial in $a_0,\ldots, a_k$ with constant coefficients (see \cite[Page 405, (1.30)]{GKZ}).
\end{proof}

\begin{proof}[Proof of Corollary \ref{cor}]
    For a K3 surface $X$, since the canonical bundle $K_X$ is trivial, there is a global holomorphic frame $\sigma$ of $K_X$. According to Lemma \ref{lem long}, for $k\geq 2$, the section $x\mapsto \Delta^k_{\sigma(x)}$ is nonzero vector in $H^0(X, S^{2(k-1)}(S^kT^{1,0}X)^*)$. Therefore by Lemma \ref{line bundle},  $S^kT^{1,0}X$ is not RC-positive for $k\geq 2$.
\end{proof}

\bibliographystyle{amsalpha}
\bibliography{Dominion}

\textsc{June E Huh Center for Mathematical Challenges, Korea Institute for Advanced Study, 85 Hoegi-ro, Dongdaemun-gu, Seoul, Republic of Korea.} \texttt{ \textbf{ bensonchou@kias.re.kr}}

\textsc{National Dong Hwa University, No.1, Section 2, Daxue Road, Hualien, Taiwan.} \texttt{\textbf{wuuuruuu@gmail.com}}

\end{document}